\documentclass[12pt]{amsart}
\usepackage{amsmath}
\usepackage{amsfonts}
\usepackage{amssymb}
\usepackage{amsthm}
\usepackage[all]{xy}
\usepackage{color}
\usepackage{verbatim}
\usepackage{graphicx}
\usepackage{float}

\newtheorem{theorem}{Theorem}[section]
\newtheorem{lemma}[theorem]{Lemma}

\newtheorem{corollary}[theorem]{Corollary}
\newtheorem{definition}[theorem]{Definition}

\numberwithin{equation}{section}

 \author[lokenath Kundu]{Lokenath Kundu}
 
  \address{SRM, A.P.}
 
 \email{bholaktw2010@gmail.com}

    \address{SRM, A.P.}
  
\title[Growth of a Family of Finite Simple Groups] {Growth of a Family of Finite Simple Groups}

\begin{document}
\maketitle
\section{Abstract}
We consider the growth of an infinite family of finite groups. We are motivated by the remarkable contribution of Bass, Wolf, Milnor, Gromov,
Grigorchuk on the word growth and structure of infinite groups, and the results of Black on the word growth of an infinite family of finite groups. We follow the definition of the word
growth for families of finite groups as given by Black, and compute the growth of a family of finite linear fractional groups. Some developments are analogous to infinite
cases. However, contrasts are also transcribed, as well as other results.

\section{Introduction}
\noindent Let $G$ be an infinite group, and $ S=~\lbrace s_1,s_2,\dots,s_n ~\rbrace$ be a finite generating set of $G$. So, any element $g \in G$ can be expressed as $g=s_{i_1}^{\pm j_1}.s_{i_2}^{\pm j_2}\dots s_{i_k}^{\pm j_k}$ where $s_{i_1},~s_{i_2},\dots,s_{i_k} \in S$ and $j_1,j_2,\dots,j_k$ are positive integer. We consider here the cancellation free expression of $g$. Then the length of an element $l(g):= j_1+j_2+\dots+j_k$. The growth function $\gamma(n)$ for the group $G$ is the cardinality of the collection of those elements of length less than or equal to $n$. The growth series is defined as $f(z)= a_0+a_1z+\dots+a_n z^n+\dots$ where $a_n= |\lbrace g \in G| l(g)=n \rbrace|$. Clearly $\gamma(n)=a_0+a_1+\dots+a_n$. The growth of infinite group is independent of the generating set. Earler study of the growth series was exercised about the asymptotic classification of the coefficient of the growth seies \cite{polgro}. The growth of solvable groups is well understood, and in \cite{mil1, wolf} author attested that the solvable groups either has the polynomial growth or exponential growth. The polynomial growth of a group was completely characterized by Gromov in \cite{gro}. He proved that a group has polynomial growth if and only if it is virtually nilpotent. The first group which has intermediate growth $(i.e.$ the group neither has polynomial nor exponential growth) was introduced by Grigorchuk in \cite{grigo}. He constructed a finitely generated infinite group which is not finitely presented and showed that the group has intermediate growth. When we have one group which has intermediate growth, then we can construct infinitely many groups having intermediate groups, by taking direct products with itself.  In \cite{mann} we have witnessed infinitely many groups which have intermediate groeth. The theory of growth of infinite groups are well understood, and it also has a connection with the growth of a family of finite groups. \\
\noindent The growth of a family of finite groups $\mathcal{G}:=(G_i)$ with respect to some generating set $\mathcal{S}=(S_i)$ for $i \in I$  is defined as $\gamma_{\mathcal{S}}^{\mathcal{G}} ~ = ~\smash{\displaystyle \max_{i ~ \in ~ I}}(\gamma_{S_i}^{G_i})$ for all $n ~ \in ~ \mathbb{N}$. The above notion of the growth function for a family of finite groups makes sense as $\mathcal{S}$ is a bounded sized generating family and hence, for every n,  $\smash{\displaystyle \max_{i \in I}}(\gamma_{S_i}^{G^i}(n)) $ is attained. Consider the family of cyclic groups $\mathcal{G}~:=(G_i)_{i \in I}$ of order $p_i q_i$, where $p_i,~q_i$ are coprime to each other, and $q_i=p_i+1$. Now if we consider the generating family $S_i= \langle a_i\rangle$ of $G_i ~ \in \mathcal{G}$, then we have linear growth. If we consider the generating set $S_i=~ \langle ~ \lbrace x_i,y_i\rbrace ~ \rangle$ where $|x_i|=p_i$ and $|y_i|=q_i$, then we have quadratic growth. Here $|x|$ denote the order of an element in a group $G$. Contrasting to the theory of the growth of infinite groups, the growth of a family of finite groups is dependent of generating set. There is a one to one group theoretic correspondence between residually finite groups and their finite quotients. In \cite{mann1, mann3, black}, the connections has been well portraied.
In fact, we have a finitely generated residually finite group $G$ and finite generating set $S$. For a given family of groups $\mathcal{G}_i$, consisting of all finite quotients of $G$, with respective generating sets $\mathcal{S}_i$. Here $\mathcal{S}_i$ is the epimorphic image of $S$ under the mapping $G ~ \rightarrow ~ \mathcal{G}_i$. Then it is an obvious observation that $\gamma_S^G$ and $\gamma_{\mathcal{S}_i}^{\mathcal{G}_i}$ act as the same function \label{connection}. For more details on the growth of finite groups see \cite{black}. \\
\noindent One of the foundational upshots that display the neatness of the concept of the growth of group is due to A.S. Schwarz. In 1957 he proved
\begin{theorem}
Let $\tilde{M}$ be the universal cover of a compact Riemannian manifold $M$. Then
the rate of growth of $M$ is equal to the rate of growth of the fundamental group $\pi_1(M)$.
\end{theorem}
  
\noindent The growth of Riemmanian manifold $M$ is defined as the growth of the function $\gamma(r)= ~\\ Vol(B_x(r))$ as $r \rightarrow \infty$. Here $Vol(B_x(r))$ denotes the volume of a ball of radius $r$, centered at a fixed point $x \in M$. The rate of growth of Riemannian manifold is independent of the choice of the point $x \in M$. The theory of the growth of groups is profoundly associated with the other branches of mathematics like geometry, and topology and that is perspicuous from the above theorem. \\
\noindent There is a correspondene between the growth of infinite groups and a family of finite groups [\ref{connection}]. In \cite{black} author posed the following conjecture, \label{conj} "Every family of finite nonabelian simple groups is a
totally exponential family." We will address the above question partially, and give an affirmative answer for the family of non abelian finite simple groups $(PSL_2(\mathbb{F}_p))$ for $p \equiv 3~(mod ~ 4)$. $PSL_2(\mathbb{F}_p)$ is the collection of $2 \times 2$ matrices over the fields  $\mathbb{F}_{p}$ with diterminant $1$, modulo its center. The order of the group $PSL_2(\mathbb{F}_p)$ is $\frac{p(p^2-1)}{2}$. The order of the elements of the group $PSL_2(\mathbb{F}_p)$ are $p,~2,~3,~4,~\text{or}~5,~d $ and a divisor of either $\frac{p-1}{2}$ or $\frac{p+1}{2}$ where $d$ is defined as $d= min \lbrace ~ e| ~ e \geq 7 \text{ and either } e| \frac{p-1}{2} \text{ or } ~ e| \frac{p+1}{2}  \rbrace$.  
The complete list of maximal proper subgroups of $PSL_2(\mathbb{F}_p)$ is given as following \cite{sjerve} 
\begin{itemize}
\item[1.] dihedral group of order $p-1$ or $p+1$.
\item[2.] solvable group of order $\frac{p.(p-1)}{2}$.
\item[3.] $A_4$ if $p \equiv 3,13,27,37 ~ (mod ~ 40)$.
\item[4.] $S_4$ if $p \equiv \pm 1 ~ (mod ~ 8)$.
\item[5.] $A_5$ if $p \equiv \pm 1 ~ (mod ~ 5)$.
\end{itemize}
\noindent Availing one self to the above information about $PSL_2(\mathbb{F}_p)$, we compute finite generating sets of $PSL_2(\mathbb{F}_p)$ establishing  surface-kernel epimorphism map from Fuchsian group to $PSL_(\mathbb{F}_p)$.\\
\noindent Fuchsian groups are the discrete subgroups of $PSL_2(\mathbb{R})$ \cite{sve}. We consider only the cocompact Fuchsian groups $(i.e.$ the Fuchsian group $\Gamma$ has a compact fundamental region $F$). Hence, $F$ has finitely many sides, and hence finitely many vertices, finitely many elliptic cycles, and by $Theorem ~ 3.5.2$ \cite{sve}, a finite number of periods say      $m_1,~m_2,\dots,~m_r$. The quotient space $\mathcal{H}/\Gamma$ is an orbifold $(i.e.$ a compact connected orientable surface of genus $h$ with exactly $r$ marked points). Then $(h;m_1,~m_2,\dots,m_r)$ is known as the signature of the Fuchsian group $\Gamma$. In \cite{gof} authors computed that the growth function of Fuchsian groups are rational function, and it is depending on the number of edges and the angle between the edges. They also specified that the growth function is independent of the identification of the edges of the fundamental region. The growth function will allow us to compute the growth of the family of finite groups $(PSL_2(\mathbb{F}_p))$ for $p \equiv 3~(mod ~ 4)$. The signature of a finite group $G$ is defined in the following way, and it is intently associated with the notion of signature of Fuchsian groups.\\
\noindent Let $\Sigma_g$ be a compact, connected, orientable Riemann surface of genus $g~\geq ~ 0$. Any orientation preserving action of a finite group $G$ on a Riemann surface $\Sigma_g$ of genus $g$ gives an orbit space $\Sigma_h ~ := \Sigma_g/G$, and the orbit space $\Sigma_h$ is again a Riemann surface possibly with some marked points. The quotient map $p~:~\Sigma_g~\rightarrow~\Sigma_h$ is a branched covering map. Let $B=~\lbrace c_1,c_2,\dots,c_r~ \rbrace$ be the set of all branch points in $\Sigma_h$ and $A:=p^{-1}(B)$. Then  $p:~\Sigma_g \setminus A ~\rightarrow ~\Sigma_h \setminus B$ is a proper covering. The tuple $(h;m_1,m_2,\dots,m_r)$ is known as signature of the finite group $G$, where $m_1,m_2,\dots,m_r$ are the order of stabilizer of the preimages of the branch points  $c_1,c_2,\dots,c_r$ respectively. We can take this action as conformal action, that means the action is analytic in some complex structure on $\Sigma_g$, as the positive solution of Nielson Realization problem \cite{niel,eck} implies that if any group $G$ acts topologically on $\Sigma_g$ then it can also act conformally with respect to some complex structure. \\
\noindent By Riemann-Hurwitz formula we have $$ (g-1)=~|G|(h-1)+\frac{|G|}{2}\sum_{i=1}^r(1-\frac{1}{m_i}) .$$ The signature of a group encodes the information of the group action of a Riemann surface and vice versa.\\
\noindent For convenience, we write $(h;2^{[2]},3^{[1]})$ for the signature $(h;2,2,3)$. In general we use the notation $(h;~ 2^{[a_2]},~ 3^{[a_3]},~ 4^{[a_4]},~ 5^{[a_5]},~ d^{[a_d]},~ \frac{p-1}{2}^{[a_{\frac{p-1}{2}}]},~ \frac{p+1}{2}^{[a_{\frac{p+1}{2}}]},~ p^{[a_p]})$ for a signature of $PSL_2(\mathbb{F}_p)$. \\
\noindent Below we state our main theorem
\begin{theorem}\label{main}
The family of finite, simple, and non abelian groups  $(PSL_2(\mathbb{F}_p))$ has the exponential growth for $p \equiv 3~(mod 4)$.
\end{theorem} 
\noindent To prove our main theorem we have to prove the following key lemma
\begin{lemma}
$(h_{\geq ~ 0};~ 2^{[a_2]},~ 3^{[a_3]},~ 4^{[a_4]},~ 5^{[a_5]},~ d^{[a_d]},~ \frac{p-1}{2}^{[a_{\frac{p-1}{2}}]},~ \frac{p+1}{2}^{[a_{\frac{p+1}{2}}]},~ p^{[a_p]})$ is a signature for $PSL_2(\mathbb{F}_p)$ for $P ~ \equiv ~ 3 ~ (mod ~ 4)$ if and only if $$2(h-1)+~\frac{a_2-1}{2}~ + \frac{2a_3-1}{3} + ~ \frac{3a_4}{4} +~ \frac{4a_5}{5} +~ \frac{(d-1)a_d+1}{d} ~+ \frac{a_{\frac{p-1}{2}}(p-3)}{p-1} ~+ \frac{a_{\frac{p+1}{2}}(p-1)}{p+1} $$ $$+\frac{(p-1)a_p}{p} ~ \geq 0 \text{ or }$$  $$20(h-1) ~ + 10[\frac{a_2}{2} ~ +\frac{2.a_3}{3} ~+\frac{3.a_4}{4} ~+\frac{4.a_5}{5} ~+\frac{(d-1)a_d}{d} ~+\frac{(p-3)a_{\frac{p-1}{2}}}{p-1} ~+$$ $$\frac{(p-1)a_{\frac{p+1}{2}}}{p+1} ~+\frac{(p-1)a_p}{p} ] ~ \geq ~ 1 $$ when $p ~ \geq ~ 13, ~ p ~ \equiv ~ \pm ~ 1(\mod ~ 5),~ p ~ \not \equiv ~ \pm ~ 1(\mod ~ 8), ~ \text{and} ~ d \geq 15$.
\end{lemma}
\noindent The paper is organized as follows. In Section $2,$ we discuss background results. In Section $3,$ and Section $4,$ we will prove our main theorem and will show some applications. 
\section{Preliminary}
\noindent In This section we will talk about the basic definition and well known results, theorems. Mostly we will state the theorems, and results. We will also provide small proof if it is contingent. We begin this section by defining the growth of group.

\noindent Let $G$ be an infinite group, and $S ~ = ~ \lbrace s_1, ~s_2,~, \dots,~s_n\rbrace$ be a generating set of $G$. The length function of the group $G$ with respect to $S$ is defined as $$l_S(g)~ =~ \text {The minimum length of a word in the elements of} ~ S ~ \text{representing}~ G .$$
\noindent The growth function $\gamma ~ : \mathbb{N} ~ \rightarrow ~ \mathbb{N}$ of the group $G$ with respect to the generating set $S$ defined as $$\gamma_S^G(n) ~ = ~ |\lbrace ~ g ~ \in ~ G ~ |l_s(g) ~ \leq ~ n \rbrace |.$$
\noindent Two funtions $\gamma_1 ,~\gamma_2 :\mathbb{N} ~ \rightarrow ~\mathbb{N}$ are said to be equivalent iff $\gamma_1(n) ~ \leq ~ C.\gamma_2(\alpha.n)$ and $\gamma_2(n) ~ \leq ~ C.\gamma_1(\alpha.n)$ for some $C, ~ \alpha ~ \geq 0$. We will write $\gamma_1 ~ \sim ~ \gamma_2$ if $\gamma_1, ~ \text{and}~ \gamma_2$ are equivalent. The following theorem states that the growth of an infinite group is independent of the choice of generator. 
\begin{theorem}
Let $S_1,~S_2$ be two finite generating set of the group $G$. Then the growth functions $\gamma_{S_1}$, and $\gamma_{S_2}$ are equivalent. 
\end{theorem}   
\begin{proof}
Let $S_1=\lbrace e_1,e_2,\dots,e_l\rbrace$, and $S_2=\lbrace f_1,f_2,\dots,f_s\rbrace$. Then any element of $S_1$ can be represented by interms of elements of $S_2$, $i.e.$ $$e_i=f_{i1}^{\pm j1}.f_{i2}^{\pm j2}.\dots.f_{is}^{\pm js}.$$
\noindent Let $\alpha = \smash {\displaystyle \max_{i=1}^ll(e_i)}$.\\
\noindent Now the proof follows immediately.
\end{proof}
\noindent Now we are going to define the growth of different types of infinte group.
\begin{definition}{(Polynomial growth)}
Let $G$ be a group, and $\gamma$ be the growth function of $G$. Then $G$ is said to have polynomial growth if $\gamma (n)~ \sim n^{\alpha}$ for some $\alpha \geq 0$. 
\end{definition}
\begin{definition}{(Exponential growth)} Let $G$ be a group, and $\gamma$ be the growth function of $G$. Then $G$ is said to have exponential growth if $\gamma (n)~ \sim e^n$ for.
\end{definition}
\begin{definition}{(Superpolynomial growth)}
Let $G$ be a group, and $\gamma$ be the growth function of $G$. Then $G$ is said to have polynomial growth if $\displaystyle \lim_{n \to \infty} \frac{ln(\gamma (n))}{ln ~ n}= \infty$. 
\end{definition}
\begin{definition}{(Subexponential growth)}
Let $G$ be a group, and $\gamma$ be the growth function of $G$. Then $G$ is said to have Subexponential growth if $\displaystyle \lim_{n \to \infty} \frac{ln(\gamma (n))}{n}= 0$. 
\end{definition}
\begin{definition}{(Intermediate growth)}
A group $G$ is said to have intermediate growth if $G$ has both subexponential and superpolynomial growth.
\end{definition}
\noindent The polynomial growth of an infinite group is completely represented by the next theorem.
\begin{theorem}
A finitely generated group $G$ has polynomial growth if and only if $G$ is virtually nilpotent. 
\end{theorem}
\begin{proof}
\cite{gro}.
\end{proof}

\section{Word Growth of Finite Group}
\noindent Previously we studied the growth of infinite groups with bounded set of generators. In this section we will study the analogous case, where we consider a infinite family of finite groups with uniformly bounded sized finite generating set.
\begin{definition}
Let $\mathcal{G} ~ =~ (G_i)_{i ~ \in ~ I}$ be a family of $k-$ generated finite groups. Let $\mathcal{S}~=~(S_i)_{i~\in ~I}$  be a family of respective generating set for $\mathcal{G} ~ =~ (G_i)_{i ~ \in ~ I}$ with $|S_i|~ \leq ~ k$ for all $i ~ \in ~ I$. Abusing the notations, we shall henceforth call such generating families a bounded sized generating families. The growth function of $\mathcal{G}$ with respect to the generating set $\mathcal{S}$ is defined as  $\gamma_{\mathcal{S}}^{\mathcal{G}} ~ = ~ \displaystyle \max_{i ~ \in ~ I}(\gamma_{S_i}^{G_i})$ for all $n ~ \in ~ \mathbb{N}$.
\end{definition} 
\noindent The above definition makes sense , as $\mathcal{S}$ is a bounded sized generating family, so for every $n,~ \displaystyle \max_{i~ \in ~ I}(\gamma_{S_i}^{G_i}(n))$ is attained. In this finite case, it can be seen that $\displaystyle \lim_{n ~ \to ~ \infty}\gamma_{\mathcal{S}}^{\mathcal{G}} (n)^{\frac{1}{n}}$ exists, and it is finite. So the analogy of polynomial growth, exponential growth, super polynomial growth, and sub exponential growth are also make sense here. The corresponding definitions are follows.

\begin{definition}{(Totally exponential family)}
A family of groups $\mathcal{G}$ is said to be totally exponential family if $\mathcal{G}$ grows exponentially with respect to every bounded sized generating family. In similar fashion we can define the totally polynomial family or intermediate family of groups.
\end{definition}
\noindent There is a one to one group theoretic correspondence between residually finite groups and their finite quotients. In \cite{mann1, mann3, black} the connections was well portraied.
In point of fact, we have a finitely generated residually finite group $G$ and finite generating set $S$. For a given family of groups $\mathcal{G}=G_i$, consisting of all finite quotients of $G$, with respective generating sets $\mathcal{S}=S_i$. Here $S_i$ is the epimorphic image of $S$ under the mapping $G ~ \rightarrow ~ G_i$. Then it is an obvious observation that $\gamma_S^G$ and $\gamma_{\mathcal{S}}^{\mathcal{G}}$ act as the same funtion \ref{connection}. The following theorem characterize the polynomial growth for a family of finite groups.
\begin{theorem}
Let $(\mathcal{G},~ \mathcal{S})$ be a $k$ generator family of finite groups. Then the following statements are equivalent:
\begin{enumerate}
\item[1.] There exists constants $c,~l ~>0$ such that every $G_i ~ \in  ~ \mathcal{G}$ has a nilpotent normal subgroup of nilpotency class at most $c$ and index in $G_i$ at most l.
\item[2.] There exists a generating family $\mathcal{S}$ for $\mathcal{G}$ such that $(\mathcal{G},~ \mathcal{S})$ grows polynomially.
\item[3.] $(\mathcal{G},~ \mathcal{S})$ grows polynomially for all bounded sized generating families $\mathcal{S}$ of $\mathcal{G}$.
\end{enumerate}
\end{theorem}
\begin{proof}
\cite{black}.
\end{proof}
\noindent The next two theorems ensure the existence of a family of finite groups with intermediate grwoth, but there doesn't exists totally intermediate family of finite groups.
\begin{theorem}
There exists a family of finite group having intermediate growth with respect to some generating family.
\end{theorem}
\begin{proof}
\cite{black}.
\end{proof}
\begin{theorem}
Let $\mathcal{G}$ be a family of groups which grows subexponentially for all bounded sized generating set $\mathcal{S}$, then $\mathcal{G}$ is nil-$c$-by-index-l.
\end{theorem}
\begin{proof}
\cite{black}.
\end{proof}
\begin{corollary} We can conclude that if $\mathcal{G}$ grows subexponentially, then either $\mathcal{G}$ is a totally polynomial family or there exists a generating family $\mathcal{T}$ such that $(\mathcal{G},~ \mathcal{T})$ grows exponentially. 
\end{corollary}
\section{Fuchsian group}
\noindent A discrete subgroup of $PSL_2(\mathbb{R})$ is known as \emph{Fuchsian group} \cite{sve}. The following theorem narrates a presentation for Fuchsian group.
\begin{theorem}
{\cite{tb}} If $ \Gamma $ is a Fuchsian group with compact orbit space $ \mathcal{U}/ \Gamma $ of genus $h$ then there are elements $ \alpha_{1},\beta_{1}, \dots ,\alpha_{h},\beta_{h},c_{1},\dots,c_{r} $ in $ Aut(\mathcal{U}) $ such that the following holds,

\begin{enumerate}
\item We have $ \Gamma = \langle \alpha_{1},\beta_{1}, \dots ,\alpha_{h},\beta_{h},c_{1},\dots,c_{r} \rangle. $
\item Defining relations for $ \Gamma $ are given by $$ c_{1}^{m_{1}},\dots,c_{r}^{m_{r}},\prod_{i=1}^{h}[\alpha_{i},\beta_{i}] \prod_{j=1}^{r}c_{j} $$ where the $ m_{i} $ are integers with $ 2\leq m_{1}\leq \dots \leq m_{r}. $
\item  Each non identity element of finite order in $ \Gamma $ lies in a unique conjugate of $ \langle c_{i} \rangle $ for suitable $i$ and the cyclic subgroups $\langle c_i\rangle$ are maximal in $ \Gamma. $ 
\item Each non identity element of finite order in $ \Gamma $ has a unique fixed point in $ \mathcal{U}. $ Each element of infinite order in $ \Gamma $ acts fixed point freely on $ \mathcal{U} .$ Finite order elements are called elliptic elements and infinite order element are called hyperbolic elements.  
\end{enumerate}
\end{theorem}

\noindent For a Fuchsian group $ \Gamma $ as above, the numbers $ h,r $ and $ m_{1},m_{2},\dots,m_{r} $ are uniquely determined; $ 2h $ is the rank of free abelian part of the commutator factor of $ \Gamma. $ \\
We call $\sigma:= (h;m_{1},m_{2},\dots,m_{r}) $ the signature of $ \Gamma, $ the integer $ h $ is called the orbit genus and $ m_{i} $ are called the periods of $ \Gamma $ for $i=1,2,\dots,r. $ We define $$\Gamma(\sigma)=\langle \alpha_{1},\beta_{1}, \dots ,\alpha_{h},\beta_{h},c_{1},\dots,c_{r}\\ |
c_{1}^{m_{1}}=\dots,=c_{r}^{m_{r}}=\prod_{i=1}^{h}[\alpha_{i},\beta_{i}] \prod_{j=1}^{r}c_{j}=1\rangle $$ 

\noindent Let $G$ be a group, and $g  \in  G$ be a non identity element. If there exists a finite homomorphic image $G^*$ of $G$ such that the image $g^*$ of $g$ is not identity, then $G$ is said to be residually finite. In particular Fuchsian groups are residually finite \cite{res}. This notion will allow us to compute the growth of the family of finite groups $\lbrace ~PSL_2(\mathbb{F}_p)~\rbrace$ where $p ~ \equiv ~3~(mod ~ 4)$.

\noindent Now we define the signature of a finite group in the sense of Harvey \cite{har}.

\begin{lemma}[Harvey condition]
\label{Harvey condition}
A finite group $G$ acts faithfully on $\Sigma_g$ with signature $\sigma:=(h;m_1,\dots,m_r)$ if and only if it satisfies the following two conditions: 

\begin{enumerate}

\item The \emph{Riemann-Hurwitz formula for orbit space} i.e. $$\displaystyle \frac{2g-2}{|G|}=2h-2+\sum_{i=1}^{r}\left(1-\frac{1}{m_i}\right), \text{ and }$$

\item  There exists a surjective homomorphism $\phi_G:\Gamma(\sigma) \to G$ that preserves the orders of all torsion elements of $\Gamma$. The map $\phi_G$ is also known as surface-kernel epimorphisom.
\end{enumerate}
\end{lemma}

\noindent Suppose we have a signature $(h;m_{1},m_{2},\dots,m_{r})$ for a given group $G$. Extension principle gives us a clear idea when $(h;m_{1},m_{2},\dots,m_i^{\prime},m_i^{\prime \prime},\dots,m_{r})$ is again a signature of the group G.
\subsection{Extension principle}
\begin{theorem}
Let $(h;m_{1},m_{2},\dots,m_{r})$ be a signature of a finite group $G=~PSL_2(\mathbb{F}_p)$ where $p$ is a prime number. Now $$\Gamma(h;m_{1},m_{2},\dots,m_{r})=\langle \alpha_{1},\beta_{1}, \dots ,\alpha_{h},\beta_{h},c_{1},\dots,c_{r} | c_{1}^{m_{1}},\dots,c_{r}^{m_{r}},\prod_{i=1}^{g}[\alpha_{i},\beta_{i}] \prod_{j=1}^{r}c_{j}\rangle. $$
If $c_i=c_{i1}c_{i2} \in ~ \Gamma; i \in \lbrace1,2,\dots,r\rbrace$ where $|c_{i1}|=m_i$ and $|c_{i2}|=m_i$ where $m_i$ is odd prime or $m_i=2$, then $$(h;m_{1},m_{2},\dots,m_{i},m_{i},\dots,m_{r})$$ is also a signature of $G$. It is known as \emph{extension principle}.
\end{theorem}

\begin{proof}
Let $(h;m_{1},m_{2},\dots,m_{r})$ be a signature of  finite group $G$. Then there exists a surface kernel epimorphism $$\phi: ~ \Gamma(h;m_{1},m_{2},\dots,m_{r}) ~ \rightarrow ~ G ~ \text{defined as } $$ $$ \phi (\alpha_i)=A_i$$ $$\phi (\beta_i)=B_i$$ $$\phi(c_i)=C_i.$$
As $c_i=c_{i1}c_{i2}$ where $|c_{i1}|=m_{i}$ and $|c_{i2}|=m_{i}$. Now $$\phi(c_i)=C_i ~ \Rightarrow ~ \phi(c_{i1}c_{i2})=C_i ~ \Rightarrow ~ \phi(c_{i1})\phi(c_{i1})=C_i$$ $$\Rightarrow C_{i1}C_{i2}= C_i ~ \text{where}~ \phi(c_{i1})=C_{i1} ~ \text{and}~ \phi(c_{i2})=C_{i2}.$$ As $\phi$ is surface kernel epimorphism, so $|C_{i1}|=m_{i}$ and $|C_{i2}|=m_{i}$. Now we can define the following map $$\psi :~ \Gamma(h;m_{1},m_{2},\dots,m_{i1},m_{i2},\dots,m_{r}) ~ \rightarrow G ~ \text{as}$$ $$ \psi(\alpha_j)=A_j, ~ \psi(\beta_j)=B_j,~ \psi(c_j)=C_j ~ \text{for} ~ j \neq i $$  and $$ \psi(c_{i1})=C_{i1}, ~ \text{and}~\psi(c_{i2})=C_{i2}.$$ Clearly $\psi $ is surface kernel epimorphism. \\
\noindent Now we have to show that 
$$1+|G|(h-1)+\frac{|G|}{2}\sum_{i=1}^r (1-\frac{1}{m_i})+\frac{|G|}{2}(1-\frac{1}{m_i})$$ is an integer.
As $(h;m_1,m_2,\dots,m_r)$ is a signature of $|G|$. So $$1+|G|(h-1)+\frac{|G|}{2}\sum_{i=1}^r (1-\frac{1}{m_i})$$ is an integer. Now $\frac{|G|}{2}(1-\frac{1}{m_i})$ is an integer if $m_i=2$ or $m_i$ is an odd prime dividing $\frac{(p-1)p(p+1)}{2}$. Hence $(h;m_{1},m_{2},\dots,m_{i},m_{i},\dots,m_{r})$ is a signature of $G$.
\end{proof}
\noindent Let $\Gamma$ be a Fuchsian group, and $\mathbb{H}/ \Gamma$ has finite area. The fundamental domain $F$ for the group $\Gamma$ defined as 
\begin{definition}
A fundamental domain $F$ for the group $\Gamma$ is an open subset of $\mathbb{H}$ which satisfy the followinf two condition
\begin{enumerate}
\item  $ \underset{\alpha \in \Gamma}{\cup} ~\alpha(cl(F))= \mathbb{H}$.
\item The images $\alpha(F)$ are pairwise disjoint $i.e. ~ \alpha_i(F)\cap \alpha_j(F)= \emptyset$ if $\alpha_i, \alpha_j \in \Gamma$, and $i \neq j$.
\end{enumerate}
\end{definition}
\noindent Thus F is a fundamental domain if every point lies in the closure of some image $\alpha(F)$ and
if two distinct images do not overlap. We say that the images of $F$ under $\Gamma$ tessellate $\mathbb{H}$.
\begin{lemma}
Let $\Gamma$ be a non-trivial Fuchsian group. Then there exists a point $p \in \mathbb{H}$ that is not a fixed point for any non-trivial element of $\Gamma$.
\end{lemma}
\begin{proof}
\cite{sve}.
\end{proof}
Let $\Gamma$ be a co-compact Fuchsian group, and $p \in \Gamma$ be a point which is not fixed by any point of $\Gamma$. We define the Dirichlet region to be:
$$D=~\lbrace z \in \mathbb{H}| d_{\mathbb{H}}(z,p)<d_{\mathbb{H}}(z,\alpha(p)),~ \forall~ \alpha \in \Gamma ~\backslash \lbrace id \rbrace \rbrace.$$ 

\noindent As we always consider co-compact Fuchsian group, so the Dirichlet region $D$ always convex hyperbolic polygon having finitely many edges. Now we give a description how the Dirichlet region actuate a presentation for the Fuchsian group $\Gamma$ with generators $\Sigma$. First we consider the group $\Gamma$ is torsion free, $i.e. ~ \Gamma$ is a fundamental group of a closed surface. We denote the vertex set of $D$ as $V$. Two vertices $v,w$ of $V$ are equivalent if there is a finite sequence $v= v_0,v_1,\dots ,v_n=w $ of vertices of $D$ and elements $g_1,g_2,\dots ,g_n \in \Sigma $ with $g_j v_{j-1}=v_j, ~ \forall j \in \lbrace 1,2,\dots,n\rbrace$. If $v~ \text{and}~ w$ are equivalent then we denote it by $v \sim w$. Each equivalence class of vertices is known as vertex cycle. Let $\lbrace v_1,v_2,\dots,v_m \rbrace$ consist of one vertex from each equivalence class. Let $c_i$ for $i \in \lbrace 1,2,\dots,m\ rbrace$, be the word in $\Sigma$ determined by an edge-path in the $1-$skeleton of the dual tessellation which traverses the link of $v_i$. In case that $a_0 = v_i,~ a_1= g_1 a_0,~ \dots , ~ a_r =g_ra_{r-1},~a_0=g_{r+1}a_r$ is the minimal listings of the vertex through $v_i$, then the corresponding word is $g_{r+1}g_r\dots g_2g_1$. Then the Fuchsian group $\Gamma$ has the presentation $\Gamma= \langle \Sigma |c_1,c_2,\dots,c_m \rangle$. Now the choice of different initial edge for the edge path would yield a cyclic permutation of $c_i$ and changing the orientation of the edge-path would replace $c_i$ by its inverse. In matter-of-course $\Gamma$ may have torsion elements. We assume that $\Sigma$ does not have any element of order $2$, hence each element of $\Sigma$ pairs two different edges. So $D$ is an even-sided polygon in $\mathbb{H}$. As above we follow the edge pairings to compute the vertex cycles and words $c_1,c_2,\dots ,c_m$. We also assign the integer $n_i \in \mathbb{N}$ to each cycle as describe. For $i \in \lbrace 1,~2,~\dots,~m\rbrace$, add the interror angles in $D$ at each vertex of the cycle through $v_i$. As a result we get the total angel at the image of $v_i$, say $\overline{v_i}$ in the quptient space $\mathbb{H}/\Gamma$. As $\Gamma$ is a discrete subgroup, so the angle must have the form $\frac{2\pi}{n_i}$. If $n_i=1$, then the quotient space near $v_i$ has non sigular geometric structure. Consider $1< n_i < \infty$, then $\overline{v_i}$ is a branch point of index $n_i$, and $c_i$ is an elliptic element of order $n_i$. If $n_i=\infty$, the cycle through $v_i$ consists of vertex at $\infty$, neighbourhoods of the vertices in the cycle are glued together to form a cusp in $\mathbb{H}/\Gamma$, and $c_i$ is parabolic element of infinite order. Then $\Gamma$ has the presentation $\langle \Sigma| (c_1)^{n_1},\dots,(c_m)^{n_m} \rangle$, where infinitely long relators may be excluded.\\
\noindent Consider regular $8$- gon with internal angel $\frac{2\pi}{24}$. We will identify the altenating edges. Then we have the Fuchsian group having signature $(2;3)$ (see the figure [\ref{(2;3)}]).
\newpage
\begin{figure}[H]
  \includegraphics[width=0.8\linewidth]{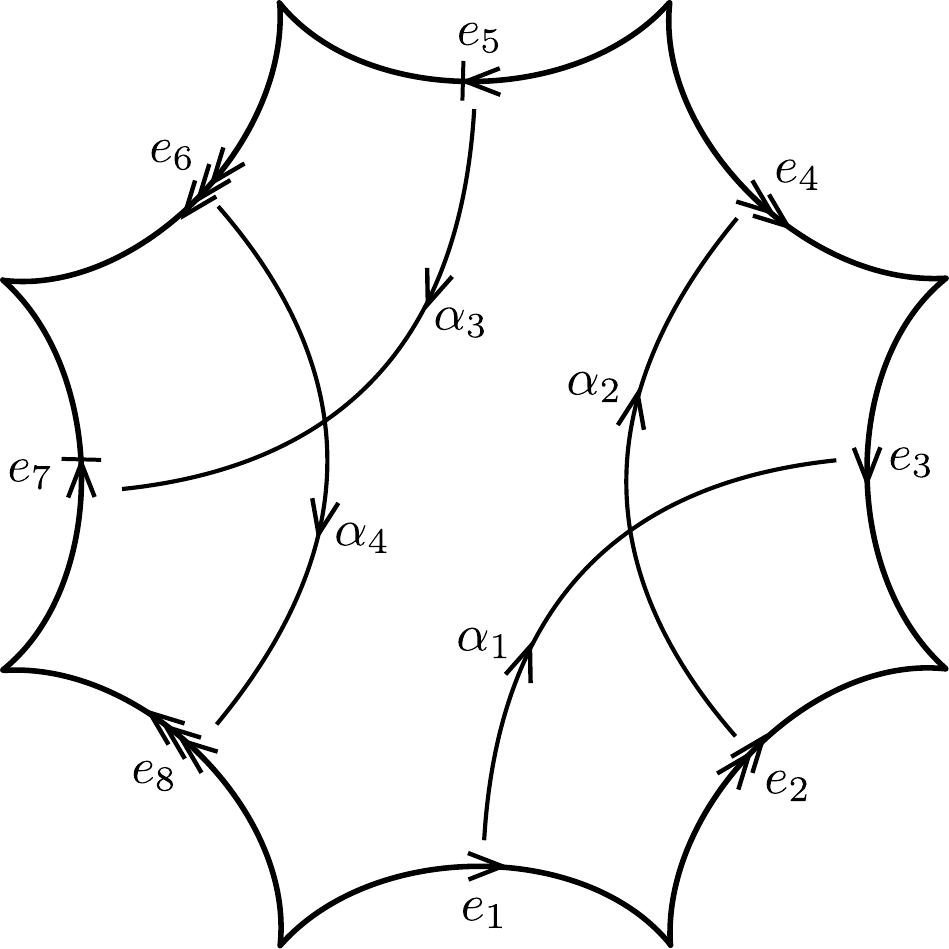}
  \caption{Regular 8-gon}
  \label{(2;3)}
\end{figure}
\noindent Now if we consider the regular $8$- gon with the internal angle $\frac{2\pi}{8}$, then we have the Fuchsian group having signature $(2;-)$ (see the figure [\ref{(2;3)}]). \\
\noindent For details see \cite{sve,elip}.

\section{Character table of $PSL_2(\mathbb{F}_p)$}

Let $G~ := GL_2(\mathbb{F}_q)$. The Borel subgroup of $G$ is denoted by $B$ is the collection of all non singular upper triangular matrix. Let $N := ~ \lbrace A = 
\begin{pmatrix}
1 & b \\
0 & 1 
\end{pmatrix}\rbrace $ and $D ~ = ~ \lbrace A = 
\begin{pmatrix}
a & 0\\
0 & d
\end{pmatrix}\rbrace \cong ~ \mathbb{F}_q^* ~ \times ~ \mathbb{F}_q^*$. Let $\mathbb{F}_q^* ~ =~ \langle \epsilon \rangle ~ s.t. ~ \sqrt{\epsilon} ~ \not\in ~ \mathbb{F}_q^*$.
Let $\mathbb{F}^{\prime}$ is a degree two extension of $\mathbb{F}_q$ and $\mathbb{F}^{\prime} ~ := ~ \mathbb{F}_q (\sqrt{\epsilon})$. So $\lbrace 1,\sqrt{\epsilon}\rbrace $ is a basis for $\mathbb{F}^{\prime}$. Let $K ~ = ~ \lbrace A = 
\begin{pmatrix}
x & y\epsilon \\
y & x
\end{pmatrix}\rbrace $ is cyclic subgroup of $G$ as $K ~ \cong ~ (\mathbb{F}^{\prime})^*$. So $|K|= q^2-1.$ Now using characteristic polynomial and minimal polynomial we easily compute the conjugacy classes of $G$. The conjugacy classes of $G$ are listed as follows

\begin{table}[ht]
\caption{Conjugacy Classes of $GL_2(\mathbb{F}_q)$} 
\centering 
\begin{tabular}{c c c} 
\hline\hline 
Representative & No. of elements in the class & Number of class \\ [0.5ex] 
\hline 
$a_x = 
\begin{pmatrix}
x & 0 \\
0 & x
\end{pmatrix}$ & $1$ & $q-1$ \\ 
$b_x = 
\begin{pmatrix}
x & 1 \\
0 & x
\end{pmatrix}$ & $q^2-1$ & $q-1$ \\
$c_{x,y} = 
\begin{pmatrix}
x & 0 \\
0 & y
\end{pmatrix}$ & $q^2+q$ & $\frac{(q-1)(q-2)}{2}$ \\
$d_{x,y} = 
\begin{pmatrix}
x & \epsilon y \\
y & x
\end{pmatrix}$ & $q^2-q$ & $\frac{(q-1)q}{2}$ 

\end{tabular}
\label{table:character table of G} 
\end{table}
We have total $q^2-1$ conjugacy classes, so we have to find $q^2-1$ irreducible character.

Let $\alpha ~ : ~ \mathbb{F}_q^* ~ \rightarrow ~ \mathbb{C}^*$ be an one dimensional character. Now $U_{\alpha} ~ : ~ G ~ \rightarrow ~ \mathbb{C}^*$ defined as $$U_{\alpha}(\begin{pmatrix}
a & b \\
c & d
\end{pmatrix}) ~ = ~ \alpha(ad-bc).$$
Let $V$ be the complementary $q$ dimensonal representation of the permutation representation of $G$ on $\mathbb{P}^1(\mathbb{F}_q)$. Let $V_{\alpha} ~ := V ~ \otimes U_{\alpha}$. The character value of this two representations is given by
\begin{table}[ht]
\caption{Character value of the representations $U_{\alpha} ~ \text{and} ~ V_{\alpha}$} 
\centering 
\begin{tabular}{c c c c c} 
\hline\hline 
Representative & $a_x$ & $b_x$ & $c_{x,y}$ & $d_{x,y}$ \\ [0.5ex] 
\hline 
$U_{\alpha}$ & $\alpha(x)^2$ & $\alpha(x)^2$ & $\alpha(x)\alpha(y)$ & $\alpha(x^2-\epsilon y^2)$ \\ 
$V_{\alpha}$ & $q\alpha(x)^2$ & $0$ & $\alpha(x)\alpha(y)$ & $-\alpha(x^2-\epsilon y^2)$ 

\end{tabular}
\label{table:character value of U and V} 
\end{table}
We use $U_{\alpha} ~ \text {and} ~ V_{\alpha}$ for the representation and as well as the character values of the corresponding representations. It is easy to check that the above two characters are irreducible. Now we have total $2q-2$ irreducible character. Now we consider $\alpha, ~ \beta ~ : \mathbb{F}_q^* ~ \rightarrow ~ \mathbb{C}^*$ two one dimensonal character. We have $$B ~ \rightarrow ~ B/N ~ = ~ D ~ = ~ \mathbb{F}_q^* ~ \times ~ \mathbb{F}_q^* ~ \rightarrow ~ ~ \mathbb{C}^* ~ \times ~ ~ \mathbb{C}^* ~ \rightarrow ~ ~ \mathbb{C}^*,$$ and $\alpha, \beta(\begin{pmatrix}
a & b \\
0 & d
\end{pmatrix})= ~ \alpha(a)\beta(d)$. Let $W_{\alpha, \beta}$ be the representation induced from $B$ to $G$ of the above representation $\alpha,\beta$. Then the character value $W_{\alpha,\beta}$ is 
\begin{table}[ht]
\caption{Character value of the representations $W_{\alpha,\beta}$} 
\centering 
\begin{tabular}{c c c c c} 
\hline\hline 
Representative & $a_x$ & $b_x$ & $c_{x,y}$ & $d_{x,y}$ \\ [0.5ex] 
\hline 
$W_{\alpha,\beta}$ & $(q+1)\alpha(x)\beta(x)$ & $\alpha(x)\beta(x)$ & $\alpha(x)\beta(y) ~ + ~ \beta(x)\alpha(y)$ & $0$ 

\end{tabular}
\label{table:character value of W} 
\end{table}

We already observe that $K \cong (\mathbb{F}^{\prime})^*$. So for any one dimensonal representation  $\phi : K \rightarrow  ~ (\mathbb{C}^*)$ we have the induced representation $Ind(\phi)$ of dimension of dimension $q^2-1$. We denote the character value of the induced representation $Ind(\phi)$ by $\phi^G$ and  it's value is given by
\begin{table}[ht]
\caption{Character value of the representations $Ind(\phi )$} 
\centering 
\begin{tabular}{c c c c c} 
\hline\hline 
Representative & $a_x$ & $b_x$ & $c_{x,y}$ & $d_{x,y}$ \\ [0.5ex] 
\hline 
$Ind(\phi )$ & $q(q-1)\phi(x)$ & $0$ & $0$ & $\phi(\zeta )+\phi(\zeta )^q $ 

\end{tabular}
 \label{table:character value of Ind }
\end{table}

Consider the character
$$\chi_{\phi} ~ = ~ V ~ \otimes ~ W_{\alpha,1} ~ - ~ W_{\alpha,1} ~ - ~ Ind(\phi ).$$ The character value of $\chi _{\phi }$ is given as follows
\begin{table}[ht]
\caption{Character value of the representations $\chi_{\phi }$} 
\centering 
\begin{tabular}{c c c c c} 
\hline\hline 
Representative & $a_x$ & $b_x$ & $c_{x,y}$ & $d_{x,y}$ \\ [0.5ex] 
\hline 
$\chi_{\phi }$ & $(q-1)\alpha(x)$ & $-\alpha(x)$ & $0$ & $-(\phi(\zeta )+\phi(\zeta )^q) $ 

\end{tabular}
 \label{table:character value of chi }
\end{table}
So the character table of $GL_2(\mathbb{F}_q)$ is as follows
\begin{table}[ht]
\caption{Character table of $GL_2(\mathbb{F}_q)$} 
\centering 
\begin{tabular}{c c c c c} 
\hline\hline 
Representative & $a_x$ & $b_x$ & $c_{x,y}$ & $d_{x,y}$ \\ [0.5ex] 
\hline 
$U_{\alpha}$ & $\alpha(x)^2$ & $\alpha(x)^2$ & $\alpha(x)\alpha(y)$ & $\alpha(x^2-\epsilon y^2)$ \\ 
$V_{\alpha}$ & $q\alpha(x)^2$ & $0$ & $\alpha(x)\alpha(y)$ & $-\alpha(x^2-\epsilon y^2)$ \\
$W_{\alpha,\beta}$ & $q+1\alpha(x)\beta(x)$ & $\alpha(x)\beta(x)$ & $\alpha(x)\beta(y) ~ + ~ \beta(x)\alpha(y)$ & $0$ \\
$\chi_{\phi }$ & $(q-1)\alpha(x)$ & $-\alpha(x)$ & $0$ & $-(\phi(\zeta )+\phi(\zeta )^q) $

\end{tabular}
\label{table:character value of G} 
\end{table}

At this moment we are considering the group $SL_2{(\mathbb{F}_q)}$. It is a subgroup of $GL_2{(\mathbb{F}_q)}$ conataining all $2 ~ \times ~ 2$ matrices of diterminant $1$. We are taking into account the subgroup $C ~ := \lbrace \xi \in ~ (\mathbb{F}^{\prime})^* ~ | \xi ~ ^{q+1} =1 \rbrace$ is a cyclic subgroup of $(\mathbb{F}^{\prime})^*$. The conjugacy classes, the  representatives of the corresponding class, and the number of each conjugacy class type are listed below

\begin{table}[ht]
\caption{Conjugacy classes of $SL_2(\mathbb{F}_q)$} 
\centering 
\begin{tabular}{c c c} 
\hline\hline 
Representative & $Number ~ of ~ elements ~ in ~ the ~ class$ & $Number ~ of ~ classes $ \\ [0.5ex] 
\hline 
$id = 
\begin{pmatrix}
1 & 0 \\
0 & 1
\end{pmatrix}$ & $1$ & $1$ \\
$-id = 
\begin{pmatrix}
-1 & 0 \\
0 & -1
\end{pmatrix}$ & $1$ & $1$ \\ 
$ 
\begin{pmatrix}
1 & 1 \\
0 & 1
\end{pmatrix}$ & $\frac{q^2-1}{2}$ & $1$ \\
$ 
\begin{pmatrix}
1 & \epsilon \\
0 & 1
\end{pmatrix}$ & $\frac{q^2-1}{2}$ & $1$ \\
$ 
\begin{pmatrix}
-1 & 1 \\
0 & -1
\end{pmatrix}$ & $\frac{q^2-1}{2}$ & $1$ \\
$ 
\begin{pmatrix}
-1 & \epsilon  \\
0 & -1
\end{pmatrix}$ & $\frac{q^2-1}{2}$ & $1$ \\
$ 
\begin{pmatrix}
x & 0 \\
0 & x^{-1}
\end{pmatrix},~ x ~ \neq ~1$ & $q(q+1)$ & $\frac{q-3}{2}$ \\
$ 
\begin{pmatrix}
x & \epsilon y \\
y & x
\end{pmatrix}, ~ x ~ \neq ~1$ & $q(q-1)$ & $\frac{q-1}{2}$ 
\end{tabular}
\label{table:conjugacy class of SL} 
\end{table} 

So, the total number of conjugacy class is $q+4$. We have to fond out $q+4$ irreducible characters. We have 
\begin{enumerate}
\item[1.] The restriction of $U_{\alpha}$ always gives us trivial representation. Hence it is irreducible.
\item[2.] The restriction of $V_{\alpha}$ yield $V$ which is irreducible.
\item[3.] The restriction of $W_{\alpha}$ of $W_{\alpha,1}$ is irreducible if $\alpha^2 ~ \neq ~ 1,$ and $W_{\alpha} ~ = ~ W_{\beta}$ for $\alpha ~ = ~ \beta ~ \text{or}~ \alpha ~ = ~ \beta^{-1}$. So we have $\frac{q-3}{2}$ irreducible character of dimension $q+1$.
\item[4.] If $ \tau^2 ~ = ~ 1$ and $\tau ~ \neq ~ 1$ then $W_{\tau} ~ = ~ W_{tau}^{\prime} ~ \oplus ~ W_{tau}^{\prime \prime}$ where $W^{\prime} ~ \text{and} ~ W^{\prime \prime}$ are irreducible character of dimension $\frac{q+1}{2}$.
\item[5.] The restriction of $\phi$ on the subgroup $C$ gives $\frac{q-1}{2}$ more irreducible character $\chi_{\phi}$ of dimension $q-1$ when $\phi^{2} ~ = ~ \neq 1$.
\item[6.] For $\psi^2 ~ = ~ 1$ and $\phi ~ \neq ~ 1$, the character is sum of two irreducible character $\chi_{\psi}= ~ \chi_{\psi}^{\prime} ~ \oplus ~\chi_{\psi}^{\prime \prime}$ of dimension $\frac{q-1}{2}$.  
\end{enumerate}
Altogether this gives us $q+4$ distinct irreducible characters So this is the complete list of characters.

Hence the following is the  character table of $SL_2{\mathbb{F}_q}$.
\begin{table}[ht]
\caption{Character table of $SL_2(\mathbb{F}_q)$} 
\centering 
\resizebox{\columnwidth}{!}{
\begin{tabular}{c c c c c c c c c} 
\hline\hline 
Representative & $\begin{pmatrix}
1 & 0 \\
0 & 1
\end{pmatrix}$ & $\begin{pmatrix}
-1 & 0 \\
0 & -1
\end{pmatrix}$ & $\begin{pmatrix}
1 & 1 \\
0 & 1
\end{pmatrix}$ & $\begin{pmatrix}
1 & \epsilon  \\
0 & 1
\end{pmatrix}$ & $\begin{pmatrix}
-1 & 1 \\
0 & -1
\end{pmatrix}$ & $\begin{pmatrix}
-1 & \epsilon  \\
0 & -1
\end{pmatrix}$ & $\begin{pmatrix}
x & 0 \\
0 & x
\end{pmatrix}, ~ x ~ \neq ~1$ & $\begin{pmatrix}
x & \epsilon y \\
y & x
\end{pmatrix}, ~ x ~ \neq ~1$ \\ [0.5ex] 
\hline 
$U$ & $1$ & $1$ & $1$ & $1$ & $1$ & $1$ & $1$ & $1$ \\ 
$V$ & $q$ & $q$ & $0$ & $0$ & $0$ & $0$ & $1$ & $-1$ \\
$W_{\alpha}$ & $q+1$ & $(q+1) \alpha (-1)$ & $1$ & $1$ & $\alpha (-1)$ & $\alpha (-1)$ & $\alpha (x)+ \alpha (x^{-1})$ & $0$ \\
$W_{\tau}^{\prime}$ & $\frac{q+1}{2} $ & $\frac{q+1}{2} \tau (-1)$ & $s$ & $t$ & $s^{\prime}$ & $t^{\prime}$ & $\frac{\tau(x) ~ + \tau(x)^{-1}}{2}$ & $0$ \\
$W_{\tau}^{\prime \prime}$ & $\frac{q+1}{2} $ & $\frac{q+1}{2} \tau (-1)$ & $t$ & $s$ & $t^{\prime}$ & $s^{\prime}$ & $\frac{\tau(x) ~ + \tau(x)^{-1}}{2}$ & $0$ \\
$\chi_{\phi}$ & $q-1$ & $q-1 \alpha(-1)$ & $-1$ & $-1$ & $ - \phi(-1)$ & $- \phi(-1)$ & $0$ & $-(\phi(\zeta) ~ + ~ \phi(\zeta))^{q})$ \\
$\chi_{\psi}^{\prime}$ & $\frac{q-1}{2}$ & $\frac{q-1}{2} \tau(-1)$ & $u$ & $v$ & $u^{\prime} $ & $v^{\prime}$ & $0$ & $-\frac{\psi(\zeta) ~ + \psi(\zeta))^{q}}{2}$ \\
$\chi_{\psi}^{\prime \prime}$ & $\frac{q-1}{2}$ & $\frac{q-1}{2} \tau(-1)$ & $v$ & $u$ & $v^{\prime} $ & $u^{\prime}$ & $0$ & $-\frac{\psi(\zeta) ~ + \psi(\zeta))^{q}}{2}$

\end{tabular}}
\label{table:character value of SL(2,q)} 
\end{table}

Now we are going to explain the value of $s, ~ t, ~ s^{\prime}, ~ t ^{\prime}, ~ u, ~ v, ~ u^{\prime}, ~ v^{\prime} $. \\
Let $\omega ~ = ~ \tau(-1) ~ i.e. ~ \omega ~ = ~ -1 ~ \text {when } ~ q ~ \equiv ~ 3 ~ (mod4)$. Now $$ s,~ t ~ = ~ \frac{-1 \pm \iota q}{2}, \text{and} ~ s^{\prime}~ = ~ \tau(-1)s, ~ t^{\prime}~ = ~ \tau(-1)t.$$
Similarly $\psi(-1)=\iota$ when $ q ~ \equiv ~ 3 ~ (mod4)$, and $$u ,~ v ~ = ~ \frac{-1 \pm \iota q}{2}, \text{and} ~ u^{\prime}~ = ~ \tau(-1)u, ~ v^{\prime}~ = ~ \tau(-1)v.$$

We are going to state the following lemma to sum up the character table of $PSL_2(\mathbb{F}_q)$ for $q ~ \equiv ~ 3 ~ (mod ~ 4)$.
\begin{lemma}
Let $N$ be a normal subgroup of $G$.
\begin{itemize}
\item[1.] If $\chi$ is a character of $G$ and $N ~ \subseteq ~ Ker ~ \chi$, then $\chi$ is constant on the cosets of $N$ in $G$ and the function $\hat{\chi}$ defined as $\hat{\chi}(gN)~ = ~ \chi(g)$ is a character of $G/N$.
\item[2.] If $\hat{\chi}$ is a character of $G/N$, then the function defined by $\chi(g)~ = ~ \hat{\chi}(gN)$ is a character of $G$.
\item[3.] Let $Irr(G)$ be the set of all irreducible character of a group $G$. Then $\hat{\chi} ~ \in ~ Irr(G/N)$ if and only if $\chi ~ \in ~ Irr(G)$.
\end{itemize}
\end{lemma}
The next table completes the computations of the character table of the group $PSL_2(\mathbb{F}_q)$ for $q ~ \equiv ~3 ~ (mod ~ 4)$.
\begin{table}[ht]
\caption{Character table of $PSL_2(\mathbb{F}_q)$} 
\centering 
\resizebox{\columnwidth}{!}{
\begin{tabular}{c c c c c c c} 
\hline\hline 
Representative & $\begin{pmatrix}
1 & 0 \\
0 & 1
\end{pmatrix}$ & $\begin{pmatrix}
1 & 1 \\
0 & 1
\end{pmatrix}$ & $\begin{pmatrix}
1 & \epsilon  \\
0 & 1
\end{pmatrix}$  & $\begin{pmatrix}
x & 0 \\
0 & x^{-1}
\end{pmatrix}, ~ x ~ \neq ~1$ & $\begin{pmatrix}
x & \epsilon y \\
y & x
\end{pmatrix}, ~ x ~ \neq ~1$ & $\begin{pmatrix}
0 & -1 \\
1 & 0
\end{pmatrix}$ \\ [0.5ex] 
\hline 
$U$ & $1$ & $1$ & $1$ & $1$ & $1$ & $1$ \\ 
$V$ & $q$ & $0$ & $0$ & $1$ & $-1$ & $1$ \\
$W_{\alpha}$ & $q+1$  & $1$ & $1$ &  $\frac{\alpha (x)+ \alpha (x^{-1})}{2}$ & $0$ & $0$ \\
$\chi_{\phi}$ & $q-1$ & $-1$ & $-1$ & $0$ & $-(\phi(\zeta) ~ + ~ \phi(\zeta)^q)$ & $-(\phi(-1) ~ + ~ \phi(-1)^q)$ \\
$\chi_{\psi}^{\prime}$ & $\frac{q-1}{2}$ & $u$ & $v$ & $0$ & $\frac{-(\psi(\zeta) ~ + ~ \psi(\zeta)^q)}{2}$ & $\frac{-(\psi(-1) ~ + ~ \psi(-1)^q)}{2}$ \\
$\chi_{\psi}^{\prime \prime}$ & $\frac{q-1}{2}$ & $v$ & $u$ & $0$ & $\frac{-(\psi(\zeta) ~ + ~ \psi(\zeta)^q)}{2}$ & $\frac{-(\psi(-1) ~ + ~ \psi(-1)^q)}{2}$
\end{tabular}}
\label{table:character value of PSL(2,q)} 
\end{table}
For more one can follow \cite{fulton}.
\section{Signature of $PSL_2(\mathbb(F)_p)$ for $p ~ \equiv ~3 ~(mod ~ 4)$}
\begin{lemma}
$(0;~ 2^{[a_2]},~ 3^{[a_3]},~ 4^{[a_4]},~ 5^{[a_5]},~ d^{[a_d]},~ \frac{p-1}{2}^{[a_{\frac{p-1}{2}}]},~ \frac{p+1}{2}^{[a_{\frac{p+1}{2}}]},~ p^{[a_p]})$ is a signature for $PSL_2(\mathbb{F}_p)$ for $P ~ \equiv ~ 3 ~ (mod ~ 4)$ if and only if $$a_2.(1-\frac{1}{2})~+a_3.(1-\frac{1}{3})~+a_4.(1-\frac{1}{4})~+a_5.(1-\frac{1}{5})$$ $$+a_d.(1-\frac{1}{d})~+a_{\frac{p-1}{2}}.(1-\frac{2}{p-1})~+a_{\frac{p+1}{2}}.(1-\frac{2}{p+1})~+a_p.(1-\frac{1}{p}) ~ \geq ~2.$$
\label{h=0}
\end{lemma}
\begin{proof}
See \cite{oza}.
\end{proof}
\begin{lemma}
$(1;~ 2^{[a_2]},~ 3^{[a_3]},~ 4^{[a_4]},~ 5^{[a_5]},~ d^{[a_d]},~ \frac{p-1}{2}^{[a_{\frac{p-1}{2}}]},~ \frac{p+1}{2}^{[a_{\frac{p+1}{2}}]},~ p^{[a_p]})$ is a signature for $PSL_2(\mathbb{F}_p)$ for $P ~ \equiv ~ 3 ~ (mod ~ 4)$ if and only if atleast one of $a_i ~ \geq ~ 1$.
\label{h=1}
\end{lemma} 
\begin{proof}
We first prove that $(1;m)$ is a signature when $m ~ \in \lbrace ~ 2,~ 3,~ 4,~ 5,~ d,~ \frac{p-1}{2},~ \frac{p+1}{2},~ p \rbrace$.
Now $p \equiv 3 ~ (mod 4) \\ \Rightarrow p-1 \equiv 2 ~ (mod 4) \\ \Rightarrow p-1=2(2k+1); \text{ for some positive integer k } \\ \Rightarrow 2 \nmid \frac{p-1}{2}$. The above observation shows us an element of order 2 never belongs to the solvable group of order $\frac{p.p-1}{2}$.
\begin{enumerate}
\item[Case 1.] We will prove $(1;2)$ is a signature of the group $PSL_2(\mathbb{F}_p)$ for $p \equiv 3 ~ (mod4)$. Consider \[
   g=
  \left[ {\begin{array}{cc}
   0 & -1 \\
   1 & 0 \\
  \end{array} } \right]
\] 
and 
\[
   X=
  \left[ {\begin{array}{cc}
   1 & 1 \\
   0 & 1 \\
  \end{array} } \right]
\]
or \[
   X=
  \left[ {\begin{array}{cc}
   1 & \epsilon \\
   0 & 1 \\
  \end{array} } \right]
.\]
Then we have $$\sum_{\chi \in Irr(PSL_2(\mathbb{F}_p))}  ~ \frac{|\chi(X)|^2 \overline{\chi(g)}}{\chi(1)} ~ = ~ 1- ~ \frac{1}{p-1}[\overline{\phi(-1)}+\overline{\phi(-1)}^p] ~ - ~ \frac{1+p^2}{2\dot(p-1)}[\overline{\psi(-1)}+\overline{\psi(-1)}^p] ~ \neq ~ 0.$$
\noindent Hence $(1;2)$ is a signature of $PSL_2(\mathbb{F}_p)$.
\item[Case 2.] Now we consider $(1;3)$. As $3|(p-1)p(p+1), \\ \text{ so either } ~ 3| \frac{p-1}{2} ~ \text{ or } ~ 3| \frac{p+1}{2}.$
\begin{itemize}
\item[Subcase 1.] If $3~ | \frac{p-1}{2}$. We take $X$ an element of order $\frac{p+1}{2}$ and $g$ is of order $3$. Now consider \[
   X=
  \left[ {\begin{array}{cc}
   x^{\prime} & y^{\prime}\epsilon \\
   y^{\prime} & x^{\prime} \\
  \end{array} } \right]
\]
and \[
   g=
  \left[ {\begin{array}{cc}
   x & 0 \\
   0 & x^{-1} \\
  \end{array} } \right]
\] where $|X|= \frac{p-1}{2} \text{ and } |g|=3$. Now $$\sum_{\chi \in Irr(PSL_2(\mathbb{F}_p))}  \frac{|\chi(X)|^2 \overline{\chi(g)}}{\chi(1)} ~ = ~ 1+ ~ \frac{1}{q} \neq ~ 0.$$

\item[Subcase 2.]
We can take \[
   g=
  \left[ {\begin{array}{cc}
   x & y\epsilon \\
   y & x \\
  \end{array} } \right]
\] where $|g|=3$ and \[
   X=
  \left[ {\begin{array}{cc}
   x^{\prime} & y^{\prime}\epsilon \\
   y^{\prime} & x^{\prime} \\
  \end{array} } \right]
\] where $|X|= \frac{p+1}{2}$. Now $$\sum_{\chi \in Irr(PSL_2(\mathbb{F}_p))}  \frac{|\chi(X)|^2 \overline{\chi(g)}}{\chi(1)} ~ = ~ \frac{p-1}{p}- ~ \frac{1}{p-1}(|\phi(\zeta)+\phi(\zeta)^p|^2)[\overline{\phi(\omega)}+\overline{\phi(\omega)}^p]- $$
$$ \frac{1}{p-1} [\overline{\psi(\omega)}+\overline{\psi(\omega)}^p](1+Re \psi(\zeta)\overline{\psi(\zeta)}^p) \neq ~ 0.$$ 
\noindent Here $\zeta ~ = ~ x+\sqrt{\epsilon}y$ and $\omega ~ = ~ x^{\prime}+\sqrt{\epsilon}y^{\prime}$.
\item[Subcase 3.] If $3~|~ \frac{p+1}{2}$, then any element of order $3$ never belongs to the maximal solvable group of order $\frac{p(p-1)}{2}$. So we will assume \[
   X=
  \left[ {\begin{array}{cc}
   1 & 1 \\
   0 & 1 \\
  \end{array} } \right]
\]
or \[
   X=
  \left[ {\begin{array}{cc}
   1 & \epsilon \\
   0 & 1 \\
  \end{array} } \right]
\] where $|X|=p$, and 
\[
   g=
  \left[ {\begin{array}{cc}
   x & y\epsilon \\
   y & x \\
  \end{array} } \right]
\] and $|g|= ~ 3$. Then $\displaystyle\sum_{\chi \in Irr(G)} \frac{|\chi(x)|^2. \overline{\chi(g)}}{\chi(1)} = 1- ~ \frac{\overline{\phi(\zeta)}+\overline{\phi(\zeta)}^p}{p-1}- ~ \frac{1+p^2}{2(p-1)}(\overline{\psi(\zeta)}+\overline{\psi(\zeta)}^p) ~ \neq 0$.

\end{itemize}
\item[Case 3.] Consider $(1;4).$ As $4 \nmid \frac{p-1}{2}$, so $4 \nmid \frac{p.(p-1)}{2}$. So we take 
\[
   X=
  \left[ {\begin{array}{cc}
   1 & 1 \\
   0 & 1 \\
  \end{array} } \right]
\]
or \[
   X=
  \left[ {\begin{array}{cc}
   1 & \epsilon \\
   0 & 1 \\
  \end{array} } \right]
\] and 
\[
   g=
  \left[ {\begin{array}{cc}
   x & y\epsilon \\
   y & x \\
  \end{array} } \right]
\] and $|g|=4, ~ |X|=p$. Then $\displaystyle\sum_{\chi \in Irr(G)} \frac{|\chi(x)|^2. \overline{\chi(g)}}{\chi(1)} = 1- ~ \frac{\overline{\phi(\zeta)}+\overline{\phi(\zeta)}^p}{p-1}- ~ \frac{1+p^2}{2(p-1)}(\overline{\psi(\zeta)}+\overline{\psi(\zeta)}^p) ~ \neq 0$.
\item[Case 4.] Now we consider $(1;5)$.
  
\begin{itemize}
\item[Subcase 1.] First we assume $5 | p-1.$ Then $5 \nmid p+1$. \[
   X=
  \left[ {\begin{array}{cc}
   x^{\prime} & y^{\prime}\epsilon \\
   y^{\prime} & x^{\prime} \\
  \end{array} } \right]
\]
and \[
   g=
  \left[ {\begin{array}{cc}
   x & 0 \\
   0 & x^{-1} \\
  \end{array} } \right]
\] where $|X|= \frac{p+1}{2} \text{ and } |g|=5$. Now $$\sum_{\chi \in Irr(PSL_2(\mathbb{F}_p))}  \frac{|\chi(X)|^2 \overline{\chi(g)}}{\chi(1)} ~ = ~ 1+ ~ \frac{1}{q} \neq ~ 0.$$
\item[Subcase 2.] Again if $5 | p-1$, then we can take \[
   g=
  \left[ {\begin{array}{cc}
   x & y\epsilon \\
   y & x \\
  \end{array} } \right]
\] where $|g|=5$ and \[
   X=
  \left[ {\begin{array}{cc}
   x^{\prime} & y^{\prime}\epsilon \\
   y^{\prime} & x^{\prime} \\
  \end{array} } \right]
\] where $|X|= \frac{p+1}{2}$. Now $$\sum_{\chi \in Irr(PSL_2(\mathbb{F}_p))}  \frac{|\chi(X)|^2 \overline{\chi(g)}}{\chi(1)} ~ = ~ \frac{p-1}{p}- ~ \frac{1}{p-1}(|\phi(\zeta)+\phi(\zeta)^p|^2)[\overline{\phi(\omega)}+\overline{\phi(\omega)}^p]- $$
$$ \frac{1}{p-1} [\overline{\psi(\omega)}+\overline{\psi(\omega)}^p](1+Re \psi(\zeta)\overline{\psi(\zeta)}^p) \neq ~ 0.$$

\item[Subcase 3.] Now we assume $5|p+1$. So $5 \nmid p-1$.  We will take \[
   g=
  \left[ {\begin{array}{cc}
   x & y\epsilon \\
   y & x \\
  \end{array} } \right]
\] where $|g|=5$ and \[
   X=
  \left[ {\begin{array}{cc}
   1 & 1 \\
   0 & 1 \\
  \end{array} } \right]
\] or \[
   X=
  \left[ {\begin{array}{cc}
   1 & \epsilon \\
   0 & 1 \\
  \end{array} } \right]
\] where $|X|= p$.  $\displaystyle\sum_{\chi \in Irr(G)} \frac{|\chi(x)|^2. \overline{\chi(g)}}{\chi(1)} = 1- ~ \frac{\overline{\phi(z)}+\overline{\phi(z)}^p}{p-1}- ~ \frac{1+p^2}{2(p-1)}(\overline{\psi(\zeta)}+\overline{\psi(\zeta)}^p) ~ \neq 0$.
\end{itemize}
\item[Case 6.] Now we consider the case $(1;d)$ where $d$ is defined as above. If possible let $d$ is a common divisor of $\frac{p-1}{2} \text{ and } \frac{p+1}{2}$. \\
 So $p=2dl_1+1 ~ \text{ and } ~ p=2dl_1-1 \text{ for two positive integers } l_1,l_2 \\ \Rightarrow d(l_2-l_1)=1 \\ \Rightarrow d=1 $ which ia contradiction to $d \geq 7$. 
 \begin{itemize}
  \item[Subcase 1.] First we consider $d| \frac{p-1}{2}$.  \[
   g=
  \left[ {\begin{array}{cc}
   x & 0 \\
   0 & x^{-1} \\
  \end{array} } \right]
\] where $|g|= ~ d$ and \[
   X=
  \left[ {\begin{array}{cc}
   x & y\epsilon \\
   y & x \\
  \end{array} } \right]
\] where $|X| ~ = ~ \frac{p+1}{2}$. Then we have
$\displaystyle\sum_{\chi \in Irr(G)} \frac{|\chi(x)|^2. \overline{\chi(g)}}{\chi(1)} = 1+ ~ \frac{1}{p} ~ \neq 0$.
\item[Subcase 2.] Now we consider \[
   g=
  \left[ {\begin{array}{cc}
   x & y\epsilon \\
   y & x \\
  \end{array} } \right]
\] where $|g|=d$ and \[
   X=
  \left[ {\begin{array}{cc}
   x^{\prime} & y^{\prime}\epsilon \\
   y^{\prime} & x^{\prime} \\
  \end{array} } \right]
\] where $|X|= \frac{p+1}{2}$. Now $$\sum_{\chi \in Irr(PSL_2(\mathbb{F}_p))}  \frac{|\chi(X)|^2 \overline{\chi(g)}}{\chi(1)} ~ = ~ \frac{p-1}{p}- ~ \frac{1}{p-1}(|\phi(\zeta)+\phi(\zeta)^p|^2)[\overline{\phi(\omega)}+\overline{\phi(\omega)}^p]- $$
$$ \frac{1}{p-1} [\overline{\psi(\omega)}+\overline{\psi(\omega)}^p](1+Re \psi(\zeta)\overline{\psi(\zeta)}^p) \neq ~ 0.$$

\item[Subcase 3.] Now we account the case $d| \frac{p+1}{2}$. Then, \[
   g=
  \left[ {\begin{array}{cc}
   x & y\epsilon \\
   y & x \\
  \end{array} } \right]
\] where $|g|=d$ and \[
   X=
  \left[ {\begin{array}{cc}
   1 & \epsilon \\
   0 & 1 \\
  \end{array} } \right]
\] or \[
   X=
  \left[ {\begin{array}{cc}
   1 & 1 \\
   0 & 1 \\
  \end{array} } \right]
\] where $|X|= p$. So we have $\displaystyle\sum_{\chi \in Irr(G)} \frac{|\chi(x)|^2. \overline{\chi(g)}}{\chi(1)} = 1- ~ \frac{\overline{\phi(\zeta)}+\overline{\phi(\zeta)}^p}{p-1}- ~ \frac{1+p^2}{2(p-1)} \\(\overline{\psi(\zeta)}+\overline{\psi(\zeta)}^p) ~ \neq 0$.
\end{itemize}
\item[Case 7.] Consider the case $(1;\frac{p-1}{2})$.
 \begin{itemize}
 \item[Subcase 1.] First we consider \[
   g=
  \left[ {\begin{array}{cc}
   x & 0 \\
   0 & x^{-1} \\
  \end{array} } \right]
\] where $|g|= ~ \frac{p-1}{2}$ and \[
   X=
  \left[ {\begin{array}{cc}
   x & y\epsilon \\
   y & x \\
  \end{array} } \right]
\] where $|X| ~ = ~ \frac{p+1}{2}$. Then we have
$\displaystyle\sum_{\chi \in Irr(G)} \frac{|\chi(x)|^2. \overline{\chi(g)}}{\chi(1)} = 1+ ~ \frac{1}{p} ~ \neq 0$.
\item[Subcase 2.] Now we consider \[
   g=
  \left[ {\begin{array}{cc}
   x & y\epsilon \\
   y & x \\
  \end{array} } \right]
\] where $|g|=\frac{p-1}{2}$ and \[
   X=
  \left[ {\begin{array}{cc}
   x^{\prime} & y^{\prime}\epsilon \\
   y^{\prime} & x^{\prime} \\
  \end{array} } \right]
\] where $|X|= \frac{p+1}{2}$. Now $$\sum_{\chi \in Irr(PSL_2(\mathbb{F}_p))}  \frac{|\chi(X)|^2 \overline{\chi(g)}}{\chi(1)} ~ = ~ \frac{p-1}{p}- ~ \frac{1}{p-1}(|\phi(\zeta)+\phi(\zeta)^p|^2)[\overline{\phi(\omega)}+\overline{\phi(\omega)}^p]- $$
$$ \frac{1}{p-1} [\overline{\psi(\omega)}+\overline{\psi(\omega)}^p](1+Re \psi(\zeta)\overline{\psi(\zeta)}^p) \neq ~ 0.$$

 \end{itemize}
 \item[Case 8.] Now we account the case $(1;\frac{p+1}{2})$. Then, \[
   g=
  \left[ {\begin{array}{cc}
   x & y\epsilon \\
   y & x \\
  \end{array} } \right]
\] where $|g|=\frac{p+1}{2}$ and \[
   X=
  \left[ {\begin{array}{cc}
   1 & \epsilon \\
   0 & 1 \\
  \end{array} } \right]
\] or \[
   X=
  \left[ {\begin{array}{cc}
   1 & 1 \\
   0 & 1 \\
  \end{array} } \right]
\] where $|X|= p$. So we have $\displaystyle\sum_{\chi \in Irr(G)} \frac{|\chi(x)|^2. \overline{\chi(g)}}{\chi(1)} = 1- ~ \frac{\overline{\phi(\zeta)}+\overline{\phi(\zeta)}^p}{p-1}- ~ \frac{1+p^2}{2(p-1)} \\(\overline{\psi(\zeta)}+\overline{\psi(\zeta)}^p) ~ \neq 0$.
\item[Case 9.] Now we look upon the case $(1;p)$. \[
   g=
  \left[ {\begin{array}{cc}
   1 & \epsilon \\
   0 & 1 \\
  \end{array} } \right]
\] or \[
   g=
  \left[ {\begin{array}{cc}
   1 & 1 \\
   0 & 1 \\
  \end{array} } \right]
\] where $|g|= p$, and  \[
   X=
  \left[ {\begin{array}{cc}
   0 & -1 \\
   1 & 0 \\
  \end{array} } \right]
\] where $|X|= 2$. Now $\displaystyle\sum_{\chi \in Irr(G)} \frac{|\chi(x)|^2. \overline{\chi(g)}}{\chi(1)} = 1- ~ \frac{|\phi(-1)+\phi(-1)^p|^2}{p-1}+ ~ \frac{1+Re \psi(-1)\overline{\psi(-1)}^p}{p-1} ~ \neq 0$.
\end{enumerate}
\noindent Now we want to prove that $(1;2,3)$. As $(0;2,3,p)$ is a signature of $PSL_2(\mathbb{F}_p)$ \cite{oza}, so there exists a surface kernel epimorphisom $\phi ~ : ~ \Gamma (0;2,3,p) ~ \rightarrow ~ PSL_2(\mathbb{F}_{p})$ such that $\phi(c_{21})=e_{21},~ \phi(c_{31})=e_{31}, ~ \phi(c_{p1})=e_{p1}$ where $\Gamma(0;2,3,p)= ~ \langle c_{21},c_{31},c_{p1}|c_{21}c_{31}c_{p1}=1, ~ |c_{21}|=2,~ |c_{31}|=3,~ |c_{p1}|=p \rangle$. As $PSL_2(\mathbb{F}_{p})$ is non abelian, simple group. So the commutator subgroup $[PSL_2(\mathbb{F}_{p}),PSL_2(\mathbb{F}_{p})]=PSL_2(\mathbb{F}_{p})$. So $\exists ~ a,b ~ \in ~ PSL_2(\mathbb{F}_{p})$ such that $[a,b]= ~ e_{p1}$. Now we can define a map $\psi ~ : \Gamma(1;2,3) ~ \rightarrow ~ PSL_2(\mathbb{F}_{p})$ as $\psi(\alpha_1)~ = ~ a, ~ \psi(\beta_1)~ = ~ b, ~ \psi(c_{21}) ~ = ~ e_{21}, ~ \psi(c_{31}) ~ = ~ e_{31}$ where $\Gamma(1;2,3)=\langle \alpha_1,\beta_1,c_{21},c_{31}| ~ [\alpha_1,\beta_1]c_{21}c_{31}=1, ~ |c_{21}|=2, ~ |c_{31}|=3\rangle$. Clearly $\psi$ is surface kernel epimorphisom. Hence $(1;2,3)$ is a signature of $PSL_2(\mathbb{F}_{p})$. Using the above argument we can easily prove that $(1;m_1,m_2)$ is a signature of $PSL_2(\mathbb{F}_p)$ where $m_1,m_2 ~ \in ~ \lbrace ~ 2,~ 3,~ 4,~ 5,~ d,~ \\ \frac{p-1}{2},~ \frac{p+1}{2},~ p \rbrace$. Repeating the above argument and from \ref{h=0} it is now obvious that $(1;~ 2^{[a_2]},~ 3^{[a_3]},~ 4^{[a_4]},~ 5^{[a_5]},~ d^{[a_d]},~ \\ \frac{p-1}{2}^{[a_{\frac{p-1}{2}}]},~ \frac{p+1}{2}^{[a_{\frac{p+1}{2}}]},~ p^{[a_p]})$ is a signature for $PSL_2(\mathbb{F}_p)$ for $P ~ \equiv ~ 3 ~ (mod ~ 4)$ if and only if atleast one of $a_i ~ \geq ~ 1$. 

\end{proof}
\begin{lemma}
$(h_{\geq ~ 2};~ 2^{[a_2]},~ 3^{[a_3]},~ 4^{[a_4]},~ 5^{[a_5]},~ d^{[a_d]},~ \frac{p-1}{2}^{[a_{\frac{p-1}{2}}]},~ \frac{p+1}{2}^{[a_{\frac{p+1}{2}}]},~ p^{[a_p]})$ is a signature for $PSL_2(\mathbb{F}_p)$ for $P ~ \equiv ~ 3 ~ (mod ~ 4)$ if and only if atleast one of $a_i ~ \geq ~ 0$. \label{h=2}
\end{lemma}
\begin{proof}
\begin{enumerate}
\item[Case 1.] We first prove that $(h_{\geq ~ 2};-)$ is a signature of $PSL_2(\mathbb{F}_p)$. As it is clear that $(0;2,3,p)$ is a signature of $PSL_2(\mathbb{F}_p)$ from \ref{h=0}, so there exists a surface kernel epimorphisom $\phi ~ : ~ \Gamma (0;2,3,p) ~ \rightarrow ~ PSL_2(\mathbb{F}_{p})$ such that $\phi(c_{21})=e_{21},~ \phi(c_{31})=e_{31}, ~ \phi(c_{p1})=e_{p1}$ where $\Gamma(0;2,3,p)= ~ \langle c_{21},c_{31},c_{p1}|c_{21}c_{31}c_{p1}=1, ~ |c_{21}|=2,~ |c_{31}|=3,~ |c_{p1}|=p \rangle$. Now consider a map $$\psi: ~ \Gamma(h_{\geq ~ 2};-) ~ \rightarrow ~ PSL_2(\mathbb{F}_p)$$ defined as $\psi(a_1)=e_{21}, ~ \psi(a_2)=e_{31}, ~ \psi(b_i)=id; 1 ~ \leq i ~\leq h, \psi(a_i)=~ id; ~ 3 ~ \leq ~ i ~ \leq h$, where $\Gamma(h_{\geq 2};-) ~ = ~ \langle a_1,~b_1,~a_2,~b_2,~\dots~ a_h,~b_h|\prod_{i=1}^h [a_i,b_i]=1 \rangle $. Clearly $\psi$ is a surface kernel epimorphism. So $(h_{\geq ~2};-)$ is a signature of $PSL_2(\mathbb{F}_p)$.
\end{enumerate}
\noindent Now using \ref{h=0}, and \ref{h=1} we can prove easily that $(h_{\geq ~ 2};~ 2^{[a_2]},~ 3^{[a_3]},~ 4^{[a_4]},~ \\
5^{[a_5]},~ d^{[a_d]},~ \frac{p-1}{2}^{[a_{\frac{p-1}{2}}]},~ \frac{p+1}{2}^{[a_{\frac{p+1}{2}}]},~ p^{[a_p]})$ is a signature for $PSL_2(\mathbb{F}_p)$ for $P ~ \equiv ~ 3 ~ (mod ~ 4)$ if and only if atleast one of $a_i ~ \geq ~ 0$.  
\label{h=2.}
\end{proof}
\begin{lemma}{(Key lemma)}
$(h_{\geq ~ 0};~ 2^{[a_2]},~ 3^{[a_3]},~ 4^{[a_4]},~ 5^{[a_5]},~ d^{[a_d]},~ \frac{p-1}{2}^{[a_{\frac{p-1}{2}}]},~ \frac{p+1}{2}^{[a_{\frac{p+1}{2}}]},~ p^{[a_p]})$ is a signature for $PSL_2(\mathbb{F}_p)$ for $P ~ \equiv ~ 3 ~ (mod ~ 4)$ if and only if $$2(h-1)+~\frac{a_2-1}{2}~ + \frac{2a_3-1}{3} + ~ \frac{3a_4}{4} +~ \frac{4a_5}{5} +~ \frac{(d-1)a_d+1}{d} ~+ \frac{a_{\frac{p-1}{2}}(p-3)}{p-1} ~+ \frac{a_{\frac{p+1}{2}}(p-1)}{p+1} ~+\frac{(p-1)a_p}{p} ~ \geq 0$$ or $$20(h-1) ~ + 10[\frac{a_2}{2} ~ +2\frac{a_3}{3} ~+3\frac{a_4}{4} ~+4\frac{a_5}{5} ~+\frac{(d-1)a_d}{d} ~+\frac{(p-3)a_{\frac{p-1}{2}}}{p-1} ~+\frac{(p-1)a_{\frac{p+1}{2}}}{p+1} ~+\frac{(p-1)a_p}{p} ] ~ \geq ~ 1 $$ when $p ~ \geq ~ 13, ~ p ~ \equiv ~ \pm ~ 1(\mod ~ 5),~ p ~ \not \equiv ~ \pm ~ 1(\mod ~ 8), ~ \text{and} ~ d \geq 15$. 
\label{h}
\end{lemma}
\begin{proof}
The proof follows from \ref{h=0}, \ref{h=1}, \ref{h=2}.
\end{proof}

\section{Growth of the family groups $PSL_2(\mathbb{F}_p)$ when $p ~ \equiv ~ 3 ~(mod 4)$}
\noindent We are already witnessed that $(h_{\geq ~ 1};3)$ is a signature of $PSL_2(\mathbb{F}_p)$. So there exists a surface kernel epimorphism from the Fuchsian group $\Gamma(h_{\geq ~ 1};3)$ to $PSL_2(\mathbb{F}_p)$ for $p ~ \equiv ~3 ~ (mod ~4)$. First we study the Growth of the Fuchsian group $\Gamma(h_{\geq ~ 1};3)$.\\
\noindent If we consider a $(4n, ~ \frac{2\pi}{12n})$ gon and identify the alternate edges, then after identification we have a surface of genus $n$, with a cone point of order $3$. Thus the Fuchsian group that acts on the upper half plane has the signature $(n;3)$. We explicitely have the growth function for $(4n, ~ \frac{2\pi}{12n})$ gon \cite{gof}. The growth function for the $(4n, ~ \frac{2\pi}{12n})$ gon is a rational function, and it is given by $$g(z)= ~ \frac{z^{6n}+2z^{6n-1}+2z^{6n-2}+ \dots + 2z^2+2z+1}{z^{6n}+(2-4n)z^{6n-1}+(2-4n)z^{6n-2}+\dots+(2-4n)z^{2}+(2-4n)z+1}. $$ 
\noindent We tote the notation of \cite{gof}, $a_n= |~\lbrace g \in \Gamma(n;3)|l(g)=n \rbrace~|$. Then  $\gamma(n)= a_0+a_1+\dots+a_n$. Here $a_k \sim (4n)^k$, and hence $\gamma_k \sim (4n)^{k+1}$. So $\Gamma(n;3)$ has exponential growth. Hence the family $(PSL_2(\mathbb{F}_p))$ has exponential growth when $p \equiv 3~(mod ~4)$.\\

\noindent Now we consider $(4n;~\frac{2\pi}{4n})$ gon. After identifying the alternate edges we will have a surface of genus $n$. So the Fuchsian group that act on upper half plane $\mathbb{H}$ is $\Gamma(n;-)$. Then the growth function of the Fuchsian group $\Gamma(n_{\geq 2};-)$ is given by
$$g(z)= ~ \frac{z^{2n}+2z^{2n-1}+2z^{2n-2}+ \dots + 2z^2+2z+1}{z^{2n}+(2-4n)z^{2n-1}+(2-4n)z^{2n-2}+\dots+(2-4n)z^{2}+(2-4n)z+1}. $$ 
\noindent We follow the notation of \cite{gof}, $a_n= |~\lbrace g \in \Gamma(n;3)|l(g)=n \rbrace~|$. Then  $\gamma(n)= a_0+a_1+\dots+a_n$. Here $a_k \sim (4n)^k$, and hence $\gamma_k \sim (4n)^{k+1}$. So $\Gamma(n;-)$ has exponential growth. Hence the family $(PSL_2(\mathbb{F}_p))$ has exponential growth when $p \equiv 3~(mod ~4)$.
   
\noindent It is an well known fact that Fuchsian groups are residually finite group. Also the generating set for the Fuchsian group $\Gamma(n;3)$ or $\Gamma(n;-)$ is finite. The family of finite groups $PSL_2(\mathbb{F}_{p_{\equiv ~3 ~ (mod 4)}})$ consisting of finite quotients of $\Gamma(n;3)$ or $\Gamma(n;-)$ with respective generating sets $S_p$ where $S_p$ is the epimorphic image of the generating set of $\Gamma(n;3)$ or $\Gamma(n;-)$ under the mapping $\Gamma(n;3) ~ \rightarrow ~ PSL_2(\mathbb{F}_p)$ or $\Gamma(n;-) ~ \rightarrow ~ PSL_2(\mathbb{F}_p)$. Then the growth funtions of the Fuchsian group, and the family $PSL_2(\mathbb{F}_{p_{\equiv ~3 ~ (mod 4)}})$ are same, and hence exponential. 
\section{Acknowledgment}
\noindent This research was supported/partially supported by the Council of Scientific \& Industrial Research. We are thankful to our colleagues and seniors Dr. Suratno Basu, Dr. Manish Kumar Pandey, Apeksha Sanghi, Prof. V. Kannan, and Mr. Aditya tiwari who provided expertise that greatly assisted the research.
We are also grateful to Dr. Manish Kumar Pandey who moderated this paper and in that line improved the manuscript significantly.


\begin{thebibliography}{27}
\bibitem{polgro} H. Bass; The degree of polynomial growth of finitely generated nilpotent groups. Proc. Lond. Math. Soc. 25, 603-614(1972).
\bibitem{mil1} J. Milnor, Growth of finitely generated solvable groups, J. Differential Geom. 2 (1968), 447-449.
\bibitem{wolf} T. Wolf, Growth of finitely generated solvable groups and curvature of Rieman-
nian manifolds, J. Differential Geom. 2 (1968), 421-446.
\bibitem{gro} M. Gromov; Groups of polynomial growth and expanding maps, publ. Math. IHES 53(1981), 53-73.
\bibitem{grigo} R.I. Grigorghuk; On Milnors problem of group growth. Sov. Math.,Dokl. 28,23-26(1983).
\bibitem{mann} Avinoam Mann; How Groups Grow. London Mathematical Society Lecture, Note Series: 395. 
\bibitem{mann1} A. Lubotzky, and A. Mann; Residually finite groups of finite rank, Math. Proc. Cambridge Philos. Soc. 106(1989), 385-388.
\bibitem{mann3} D. Segal; Uniform finiteness conditions in residually finite groups, Proc. London Math Soc. (3) 61(1990),529-543.
\bibitem{black} Sarah Black; Asymptotic growth of finite groups. Journal of Algebra 209, 402-426(1998).
\bibitem{sjerve} Henry Glover, and Denis Sjerve; Representing $PSL_2(\mathbb{F}_p)$ on a Riemann surface of least genus. L'Enseignement Mathematique, t.31(1985,p.305-325.). http://doi.org/10.5169/seals-54572.
\bibitem{sve}Svetlana Katok; Fuchsian groups. Chicago Lectures in Mathematics. University of
Chicago Press, Chicago, IL, $1992.$
\bibitem{elip} Charles Walkden; Hyperbolic Geometry. Department of Mathematics
The University of Manchester.
\bibitem{gof} William J. Floyd, and Steven P. Plotnick; Growth functions on Fuchsian groups
and the Euler characteristic. Invent. math. 88, 1 29 (1987).
\bibitem{niel}S. Kerckhoff, The Nielsen realization problem, Ann. Math. 117 (1983), 235—265.
\bibitem{eck} B. Eckmann and H. Müller, Plane motion groups and virtual Poincare duality of dimension two, Invent.
Math. 69 (1982), 293—310.
\bibitem{tb} 
Thomas Breuer; Character and automorphism groups of compact Riemann surfaces. London mathematical society lecture note series.280. Cambridge University press.
\bibitem{har}W. J. Harvey; Cyclic groups of automorphisms of a compact Riemann surface.
 Quart.
J. Math. Oxford Ser. $(2),$ $17:86-97,$ $1966.$
\bibitem{fulton} William Fulton, and Joe Harris; Representation theory, a first course.
\bibitem{oza} Murad {\"O}zaydin, Charlotte Simmons, and Jennifer Taback;Surface Symmetries and $PSL_2(p),$ Volume 359, Number 5, May 2007, Pages 2243–2268. American Mathematical Society.
\bibitem{mann2} A. Lubotzky, and A. Mann;On groups of polynomial subgroup growth, Invent. Math. 104(1991), 521-533.
\bibitem{alx}Alexander Murray Macbeath and HC Wilkie. Discontinuous groups and birational
transformations:[Summer School], Queen's College Dundee, University of St. An-
drews. [Department of Math.], Queen's College, $1961.$
\bibitem{res} Wilhelm Magnus; Residually finite groups.
\bibitem{klein} Klein, F.; Uber die Transformationen siebenter Ordnung der elliptischen Funkionen, Math. Ann $14(1879)428-471$.
\bibitem{char} I.Martin Isaacs; Character Theory of Finite groups. Academic press, INC, 1976.
\end{thebibliography}
\end{document}